\documentclass{amsart}
\usepackage{amsmath, amscd, amssymb, amsthm}
\usepackage{bbm}
\usepackage{latexsym}
\usepackage{amsfonts}
\usepackage{graphicx}
\usepackage[all,cmtip]{xy}
\usepackage[colorlinks,linkcolor=blue,breaklinks=true,urlcolor=blue,citecolor=blue,anchorcolor=blue,pagebackref]{hyperref}%
\usepackage{geometry}
\hypersetup{
pdftitle={The Streets--Tian conjecture on compact homogeneous quotients}
}
\allowdisplaybreaks[2]
\numberwithin{equation}{section}

\newtheorem{theorem}{Theorem}
\newtheorem{lemma}{Lemma}
\newtheorem{corollary}[theorem]{Corollary}
\newtheorem{remark}{Remark}
\newtheorem{proposition}{Proposition}

\newtheorem{example}{Example}

\newcommand{\R}{{\mathbb R}}

\newcommand{\Z}{{\mathbb Z}}
\newcommand{\C}{{\mathbb C}}

\renewcommand*\backref[1]{}
\renewcommand*\backrefalt[4]{ \ifcase #1 \or (cited on page #2) \else (cited on pages #2) \fi}

\newcommand{\be}{\begin{equation}}
\newcommand{\ee}{\end{equation}}
\newcommand{\bea}{\begin{eqnarray}}
\newcommand{\eea}{\end{eqnarray}}

\DeclareMathOperator{\tr}{tr}

\def\XXint#1#2#3{{\setbox0=\hbox{$#1{#2#3}{\int}$ }
\vcenter{\hbox{$#2#3$ }}\kern-.6\wd0}}

\begin{document}

\title[Streets--Tian conjecture on homogeneous quotients]
{The Streets--Tian conjecture on compact homogeneous quotients}

\author{Kexiang Cao}
\address{Kexiang Cao. School of Mathematics and Statistics, Yangtze Normal University, Chongqing 408100, China}
\email{{caokx1214@qq.com}}

\author{Shuwen Chen}\thanks{Chen is the corresponding author. All authors contributed equally on this paper.}
\address{Shuwen Chen. School of Mathematical Sciences, Chongqing Normal University, Chongqing 401331, China}
\email{3153017458@qq.com}

\thanks{}

\renewcommand{\subjclassname}{2020 Mathematics Subject Classification}

\begin{abstract}
A Hermitian-symplectic metric is a Hermitian metric whose K\"ahler form is given by the $(1,1)$-part of a closed $2$-form. Streets-Tian Conjecture states that a compact complex manifold admitting a Hermitian-symplectic metric must be K\"ahlerian (i.e., admitting a K\"ahler metric). We study Hermitian-symplectic metrics on homogeneous complex manifolds with compact isotropy, and on their compact discrete quotients. If the transitive group is connected and unimodular, we show that an invariant Hermitian-symplectic form gives an invariant K\"ahler metric, with no lattice assumption. By averaging, the same conclusion holds for compact quotients, even if the initial metric has no invariant lift. The proof uses the top exterior power of a Hermitian-symplectic form to bound the volume functional on an invariant Aeppli slice. Applying this bound to midpoint metrics gives a uniform trace estimate. From this, volume maximization produces a balanced metric in the slice, which is also pluriclosed and hence K\"ahler. We also describe the resulting K\"ahler representative and give examples showing the role of the closed lift.
\end{abstract}

\maketitle

\tableofcontents

\markleft{Kexiang Cao and Shuwen Chen}
\markright{Streets--Tian conjecture on homogeneous quotients}

\section{Introduction}\label{sec:intro}

Let $(M^n,g)$ be a Hermitian manifold with complex structure $J$ and fundamental form $\omega$.
According to Streets and Tian \cite{ST}, a real 2-form $\Omega$ on $M$ is called a \emph{Hermitian-symplectic form} if $d\Omega=0$ and $\Omega^{1,1}>0$. Equivalently, $\Omega$ is a symplectic form such that $\Omega(v,Jv)>0$ for every nonzero real tangent vector $v$; a form of this kind is also said to \emph{tame} $J$. Write $\omega(x, y) = g(Jx, y)$ for the fundamental form of a Hermitian metric $g$. Then $g$ is \emph{Hermitian-symplectic} if $\omega$ is the $(1,1)$-part of such a form. In this case $\Omega=\omega+\alpha+\overline\alpha$, where $\alpha$ is a smooth $(2,0)$-form and $\omega$ need not be closed. The question raised in \cite{ST}, commonly known as the Streets-Tian conjecture, asks whether a compact complex manifold  carrying a Hermitian-symplectic metric must admit a K\"ahler metric.

A Hermitian metric is called \emph{pluriclosed} if $\partial\overline\partial\omega=0$, and \emph{balanced} if $d\omega^{n-1}=0$. From the decomposition of $d\Omega=0$, we see that every Hermitian-symplectic metric is pluriclosed. If a metric is both balanced and pluriclosed, then it is K\"ahler; see \cite{AI}. It is precisely when the balanced metric and the pluriclosed metric are not the same that the two existence problems cannot be treated as equivalent.

Since it was first proposed, the Streets-Tian conjecture has drawn a great deal of interest from geometers. Several special cases have now been settled, along with stronger non-existence results.
In complex dimension $1$, every Hermitian fundamental form is automatically closed, so the conclusion is immediate.
In complex dimension $2$, the conjecture holds: by combining the analysis of surfaces of Streets-Tian with the cone comparison results of Li-Zhang, one shows that a compact complex surface admitting a Hermitian-symplectic structure must be K\"ahler\cite{LZ,ST}.
For the Fujiki $\mathcal{C}$ class, Chiose showed that a compact complex manifold in this class which admits a pluriclosed metric is necessarily K\"ahler; consequently, no non-K\"ahler manifold in this class can even carry a Hermitian-symplectic metric \cite{CH}.
For twistor spaces, Verbitsky proved that a compact non-K\"ahler twistor space cannot carry any pluriclosed metric, so the Streets-Tian conjecture holds for any twistor space\cite{VER}.
For the non-K\"ahler Calabi-Yau threefolds constructed by Fu-Li-Yau from K\"ahler Calabi-Yau threefolds via conifold transitions, no pluriclosed metric exists\cite{FLY}.
For non-balanced Bismut torsion-parallel manifolds, Guo-Zheng proved that if a compact Hermitian manifold admits a non-balanced Bismut-torsion-parallel (BTP) metric, then it cannot admit any Hermitian-symplectic metric; this class includes Vaisman manifolds and the corresponding cases involving Bismut K\"ahler-like metrics\cite{GZ}.
For Chern K\"ahler-like manifolds, if a compact manifold admits a non-K\"ahler Chern K\"ahler-like metric, then it has neither a Hermitian-symplectic metric nor even an astheno-K\"ahler metric, which extends the exclusion result in the Chern-flat case\cite{GZ,DLV}.
Finally, for Vaisman manifolds, Angella-Otiman proved that a non-K\"ahler Vaisman manifold cannot admit a Hermitian-symplectic metric\cite{AO}.

A compact quotient $G/\Gamma$, equipped with a left-invariant complex structure on the Lie group $G$, is usually called a {\em Lie-complex manifold}. The known results include the following.
For complex nilmanifolds, Enrietti-Fino-Vezzoni showed that if a compact complex nilmanifold carries a Hermitian-symplectic metric, then its Lie group must be abelian, and hence it is a complex torus. Consequently, no non-toral complex nilmanifold admits a Hermitian-symplectic metric\cite{EFV}.
For completely solvable solvmanifolds, Fino-Kasuya proved that, under the assumption of a left-invariant complex structure, a Hermitian-symplectic metric exists if and only if the solvmanifold is a complex torus\cite{FK}.
For several subclasses of solvmanifolds, Fino-Kasuya-Vezzoni established the Streets-Tian conjecture in the presence of a $J$-invariant nilpotent complement, an abelian complex structure, or certain almost abelian structures; this also covers Oeljeklaus-Toma manifolds. Fino-Paradiso subsequently removed the extra restrictions in the almost abelian case, thereby extending the result to general almost abelian solvmanifolds\cite{FKV,FP}.
For splitting-type solvmanifolds, Angella-Otal-Ugarte-Villacampa proved the Streets-Tian conjecture for compact complex solvmanifolds with a left-invariant complex structure of splitting type on $G=\mathbb C\ltimes\mathbb C^{n-1}$. This is a result for a specific class of complex structures, not for all solvmanifolds\cite{AOUV}.
For $2$-step solvmanifolds, the main theorem of Chen--Zheng applies. It states that if the Lie group is $2$-step solvable, and the manifold carries a Hermitian-symplectic metric, then it admits a left-invariant K\"ahler metric. This covers many previously studied Lie-group quotient cases, but not all $3$-step solvable groups\cite{ChZ}. Here "$2$-step solvable" means that the derived algebra $[\mathfrak g,\mathfrak g]$ is abelian.

For abelian ideals of codimension $2$, if a solvable Lie algebra contains an abelian ideal $\mathfrak a$ of codimension $2$, then the Streets--Tian conjecture holds for every compatible left-invariant complex structure. Guo-Zheng first proved the case in which the ideal is $J$-invariant; Cao-Zheng treated the complementary case $J\mathfrak a\ne\mathfrak a$. Such Lie algebras are at most $3$-step solvable and need not be $2$-step solvable or completely solvable\cite{GZ,CaZ}. The result for codimension-two abelian ideals also covers some $3$-step solvable examples outside the $2$-step and completely solvable settings.

Several results for certain special cases have also been obtained concerning this conjecture. In the case of degenerate twistor deformations, Soldatenkov-Verbitsky showed that a compact holomorphically symplectic K\"ahler manifold remains K\"ahler under such a deformation. Their proof proceeds by first constructing a Hermitian-symplectic structure, then applying a bimeromorphicity argument in Teichm\"uller space. This establishes the Streets-Tian conjecture for these particular deformation families, but it does not settle the conjecture for all Hermitian-symplectic manifolds\cite{SV}.

For a variational result in an Aeppli class, Dinew-Popovici introduced a functional on the Aeppli cohomology class of a given Hermitian-symplectic metric on a $3$-dimensional Hermitian-symplectic manifold. They proved that a critical point of this functional, whenever it exists, has to be K\"ahler, and they derived a necessary cohomological obstruction to the existence of a K\"ahler metric in the same class. The key limitation is the condition ``if a critical point exists'': since the paper does not establish that such a point always exists, it does not settle the general $3$-dimensional Streets-Tian conjecture \cite{DP}.

For pluriclosed flow, Ye established that the Hermitian-symplectic property is preserved. In the Hermitian-symplectic setting, Ye also proved that any Hermitian-symplectic soliton must be a K\"ahler-Ricci soliton. These findings place restrictions on the search for canonical metrics, but preservation under the flow by itself does not guarantee that every initial Hermitian-symplectic manifold is K\"ahler\cite{YE,YED}.

The setting here also permits a nontrivial compact isotropy group. Let $G$ be a connected real Lie group, $K\subset G$ a compact subgroup, and suppose that $X=G/K$ is equipped with a $G$-invariant integrable complex structure $J$. Recall that a connected Lie group is \emph{unimodular} when its left Haar measure is also right invariant. In terms of the Lie algebra $\mathfrak g$, this amounts to $\tr(\operatorname{ad}_x)=0$ for every $x\in\mathfrak g$, where $\operatorname{ad}_x(y)=[x,y]$. Our main result is the following invariant statement.

\begin{theorem}\label{thm:invariant}
Let $G$ be a connected unimodular real Lie group and $K\subset G$ a compact subgroup. Suppose that $X=G/K$ has a $G$-invariant integrable complex structure $J$ of complex dimension $n\geq2$. If $X$ admits a $G$-invariant Hermitian-symplectic form $\Omega$, then it admits a $G$-invariant K\"ahler metric. More precisely, put $\omega=\Omega^{1,1}$. The set
$$
\mathcal S(\omega)=\{\omega+(d\xi)^{1,1}>0:\ \xi\in\Omega^1_{\R}(X)^G\}
$$
contains exactly one balanced metric, and that metric is K\"ahler. Here $\Omega^1_{\R}(X)^G$ denotes the space of smooth $G$-invariant real one-forms on $X$. Neither compactness of $X$ nor the existence of a lattice in $G$ is assumed.
\end{theorem}

We now turn to compact quotients. Let $\Gamma\subset G$ be a discrete subgroup acting freely, properly discontinuously and cocompactly on $X$. We write $M=X/\Gamma$ for the orbit space of the left action $\gamma\cdot(gK)=(\gamma g)K$; in double-coset notation, $M=\Gamma\backslash G/K$. When $K=\{e\}$, we retain the notation $M=G/\Gamma$ from \cite{CaZ,GZ,ChZ}, with the same orbit-space convention. A $G$-invariant form on $X$ descends to $M$, although the $G$-action on $X$ need not descend to $M$. By Lemma \ref{lem:average}, compactness of the quotient forces $G$ to be unimodular, and any Hermitian-symplectic form on $M$ may be replaced by invariant data on $X$. Combining this with Theorem \ref{thm:invariant}, we obtain the following corollary.

\begin{corollary}\label{cor:quotient}
Let $G$ be a connected real Lie group, $K\subset G$ a compact subgroup, and let $X=G/K$ be a complex manifold of complex dimension $n\geq2$ endowed with a $G$-invariant integrable complex structure $J$. Let $\Gamma\subset G$ be a discrete subgroup whose action on $X$ by left translations is free, properly discontinuous and cocompact. If $M=X/\Gamma$ admits a Hermitian-symplectic metric, then it admits a K\"ahler metric whose lift to $X$ is $G$-invariant.
\end{corollary}

Theorem \ref{thm:invariant} with $K=\{e\}$ also yields the following Lie algebra consequence, without requiring a compact quotient.

\begin{corollary}\label{cor:algebra}
Let $\mathfrak g$ be a real unimodular Lie algebra of dimension $2n$ with $n\geq2$, and let $J$ be an integrable complex structure on $\mathfrak g$. Suppose that $(\mathfrak g,J)$ carries a Hermitian-symplectic form $\Omega\in\Lambda^2\mathfrak g^*$, that is $d_{\mathfrak g}\Omega=0$ and $\Omega(v,Jv)>0$ for all nonzero $v\in\mathfrak g$, where $d_{\mathfrak g}$ denotes the Chevalley--Eilenberg differential. Then there exists a positive real $(1,1)$-form $\widetilde\omega$ satisfying $d_{\mathfrak g}\widetilde\omega=0$ on $(\mathfrak g,J)$. In particular, $(\mathfrak g,J)$ admits a K\"ahler inner product. No assumptions of solvability, nilpotency, special ideals, or lattices are required.
\end{corollary}

Our use of an Aeppli slice follows Kwong \cite{Kwong}, who proved that on a compact homogeneous quotient with compact isotropy, the critical points of the Hermitian volume on an invariant Aeppli slice are exactly the balanced metrics \cite{Kwong}. His Theorem F obtains a unique such point under the assumption that a balanced metric already exists. We use the same critical-point mechanism, but derive boundedness from the Hermitian-symplectic form instead. The required positivity of powers of a closed lift is already used in \cite{CaZ} and \cite{GZ}. Our additional observation is that a determinant bound on an affine positive slice, when applied at its midpoint with a fixed positive element, yields a bound for the trace, not just for the determinant. Unimodularity provides the same integration-by-parts identity in the invariant complex, even when no compact quotient exists.

The generalized volume introduced by Dinew--Popovici \cite{DP} gives another variational approach to Hermitian-symplectic metrics in an Aeppli class. When the quotient is compact, our functional coincides with the Hermitian volume on a finite-dimensional space, up to a fixed positive factor; in the noncompact setting, it is a top-degree coefficient. Our argument does not rely on a convergence theorem for a geometric flow, nor does it use the special bracket identities of \cite{CaZ,GZ,ChZ}. Those identities are still useful for explicit classifications, but here our aim is an existence argument.

Section \ref{sec:prelim} fixes the Hermitian notation and recalls two facts used later. In Section \ref{sec:invariant} we discuss invariant forms and averaging. Section \ref{sec:proof} proves the main theorem and its two corollaries. Section \ref{sec:representative} discusses the K\"ahler representative and a deformation within the Aeppli slice containing it. Two examples are given in Section \ref{sec:examples}.

\section{Preliminaries}\label{sec:prelim}

Let $(M^n,g)$ be a Hermitian manifold with complex structure $J$ and fundamental form $\omega$. Extend $g=\langle\,,\,\rangle$ complex bilinearly, and fix a local \emph{unitary frame} $e=\{e_1,\ldots,e_n\}$ of $T^{1,0}M$, with dual coframe $\varphi=\{\varphi_1,\ldots,\varphi_n\}$. Then $\omega=\sqrt{-1}\sum_i\varphi_i\wedge\overline\varphi_i$, and the volume form is $\nu_\omega=\omega^n/n!$. Our convention gives $\omega(v,Jv)=g(v,v)$ for real $v$. Frame indices range from $1$ to $n$ unless otherwise stated.

Denote by $\nabla$ the \emph{Chern connection} and by $T$ its torsion, $T(x,y)=\nabla_x y-\nabla_y x-[x,y]$. Using the notations of \cite{ChZ, GZ}, write $T(e_i,e_k)=\sum_jT^j_{ik}e_j$ and $T(e_i,\overline e_j)=0$. Let $\theta=(\theta_{ij})$ be the connection matrix, defined by $\nabla e_i=\sum_j\theta_{ij}e_j$, and put $\tau_j=\frac12\sum_{i,k}T^j_{ik}\varphi_i\wedge\varphi_k$. Viewing $\varphi$ and $\tau=(\tau_1,\ldots,\tau_n)^t$ as columns, with ${}^t$ denoting transpose, we have
\begin{equation*}
d\varphi=-{}^t\theta\wedge\varphi+\tau,
\qquad
\partial\omega=\sqrt{-1}\,{}^t\tau\wedge\overline\varphi
=\frac{\sqrt{-1}}2\sum_{i,k,j}T^j_{ik}\varphi_i\wedge\varphi_k\wedge\overline\varphi_j.
\end{equation*}
Differentiating $\omega$ gives the second identity, since the connection terms cancel. The components $T^j_{ik}$ are skew-symmetric in $i,k$ and are twice those with the same notation in \cite{SS}. The \emph{Gauduchon torsion 1-form} in \cite{GZ} is $\eta=\sum_i\eta_i\varphi_i$, where $\eta_i=\sum_kT^k_{ki}$. It satisfies
\begin{equation*}
\partial\omega^{n-1}=-\eta\wedge\omega^{n-1}.
\end{equation*}
The metric $g$ is balanced if and only if $\eta=0$, or equivalently, $d \omega^{n-1}=0$.

For a form $\rho$, write $\rho^{p,q}$ for its $(p,q)$-part. If $\rho$ is a real 2-form, then $\rho^{1,1}(x,y)=\frac12\{\rho(x,y)+\rho(Jx,Jy)\}$ and $\rho(v,Jv)=\rho^{1,1}(v,Jv)$. We write $\rho^{1,1}>0$ or $\rho^{1,1}\geq0$ when this value is positive for every nonzero real $v$ or nonnegative for every real $v$, respectively. The latter condition is called \emph{semipositivity}. If $\rho^{1,1}>0$, then $\rho$ is nondegenerate. A \emph{closed real lift} of $\omega$ is a closed real 2-form $\Omega$ with $\Omega^{1,1}=\omega$. Writing $\Omega=\omega+\alpha+\overline\alpha$, where $\alpha=\Omega^{2,0}$, and using $d=\partial+\overline\partial$, we obtain
\begin{equation}\label{eq:lift-system}
\partial\alpha=0,\qquad
\partial\omega+\overline\partial\alpha=0,
\end{equation}
and their conjugates. Applying $\overline\partial$ to the second equation gives $\partial\overline\partial\omega=0$. Thus a Hermitian-symplectic metric is pluriclosed; see \cite{ST,EFV,CaZ,GZ,ChZ}.

We will need the following positivity formula; see also \cite[proof of Proposition 1]{GZ}.

\begin{lemma}[\cite{CaZ}]\label{lem:volume}
Let $V$ be a complex vector space of complex dimension $n$, let $\omega$ be a positive real $(1,1)$-form on $V$, and let $\alpha$ be a $(2,0)$-form. Put $\Omega=\omega+\alpha+\overline\alpha$. With respect to the complex orientation,
\begin{equation}\label{eq:volume-expansion}
\Omega^n=
\sum_{r=0}^{\lfloor n/2\rfloor}
\frac{n!}{(n-2r)!(r!)^2}\,
\omega^{n-2r}\wedge\alpha^r\wedge\overline\alpha^{\,r}
\ \geq\ \omega^n.
\end{equation}
Equality holds if and only if $\alpha=0$.
\end{lemma}

The $r=1$ term in \eqref{eq:volume-expansion} is strictly
positive when $\alpha\neq0$. If $n=1$, then $\alpha=0$
automatically. We also recall the following result of Alexandrov and Ivanov \cite[Remark 1]{AI}.

\begin{lemma}[\cite{AI}]\label{lem:balanced-skt}
Let $(M^n,g)$ be a Hermitian manifold, not necessarily compact. If $d\omega^{n-1}=0$ and $\partial\overline\partial\omega=0$, then $d\omega=0$.
\end{lemma}

The conclusion follows from the pointwise identity \cite[formula (2.13)]{AI}, so it does not require compactness.

\section{Invariant forms and averaging}\label{sec:invariant}

Let $G$ be a connected real Lie group, $K\subset G$ a compact
subgroup, and $X=G/K$ a complex manifold with a $G$-invariant
complex structure $J$. Denote the Lie algebras of $G$ and $K$
by $\mathfrak g$ and $\mathfrak k$, respectively, and by
$\operatorname{Ad}$ the adjoint representation. Let $\pi:G\to X$ be the natural projection and put $o=eK$,
where $e$ is the identity of $G$.
Since $K$ is compact, we may choose an
$\operatorname{Ad}(K)$-invariant complement $\mathfrak m$
such that $\mathfrak g=\mathfrak k\oplus\mathfrak m$.
The restriction $d\pi_e|_{\mathfrak m}$ identifies
$\mathfrak m$ with $T_oX$. For a differential form $\rho$
on $X$, we write $\rho_o$ for its value at $o$ and regard
it as an alternating form on $\mathfrak m$ through this
identification. Denote by $\operatorname{pr}_{\mathfrak m}:\mathfrak g
\to\mathfrak m$ the projection with kernel $\mathfrak k$.
For $x,y\in\mathfrak m$, write
$[x,y]_{\mathfrak m}=\operatorname{pr}_{\mathfrak m}[x,y]$.

Denote the spaces of smooth real $q$-forms and complex $(p,q)$-forms on $X$ by $\Omega^q_{\R}(X)$ and $\Omega^{p,q}(X)$, and their $G$-invariant subspaces by $I^q_{\R}$ and $I^{p,q}$. Evaluation at $o$ identifies $I^q_{\R}$ with $(\Lambda^q\mathfrak m^*)^K$, where the superscript $K$ denotes invariance under the isotropy action. In particular, these spaces are finite-dimensional. The operators $d$, $\partial$ and $\overline\partial$ preserve invariance. For an invariant $q$-form $\beta$, the differential at $o$ is
\begin{equation}\label{eq:relative-d}
\begin{aligned}
(d\beta)(x_0,\ldots,x_q)
=\sum_{0\leq i<j\leq q}(-1)^{i+j}
\beta([x_i,x_j]_{\mathfrak m},x_0,\ldots,\widehat x_i,\ldots,\widehat x_j,\ldots,x_q),
\end{aligned}
\end{equation}
where $x_0,\ldots,x_q\in\mathfrak m$, and the hats indicate omitted arguments. Indeed, $\pi^*\beta$ is left invariant on $G$, right $K$-invariant and \emph{horizontal}, vanishing on tangent vectors to the fibres. The Chevalley--Eilenberg formula on $\mathfrak g$ therefore gives \eqref{eq:relative-d}. This relative differential does not require $\mathfrak m$ to be a Lie subalgebra.

Since $K$ is compact, averaging a Hermitian inner product
on $\mathfrak m$ over $K$ yields a $K$-invariant Hermitian
inner product. It extends to a $G$-invariant Hermitian
metric on $X$, and hence determines a positive
$G$-invariant volume form $\nu$. The identifications above are made at $o=eK$ and do not require a global $G$-invariant frame on $X$.

\begin{lemma}\label{lem:formal-stokes}
Suppose that $G$ is connected and unimodular, $K$ is compact, and $X=G/K$ has a $G$-invariant complex structure. Let $N=\dim_{\R}X$ and fix a positive invariant volume form $\nu$. Then $I^N_{\R}=\R\nu$ and $dI^{N-1}_{\R}=0$. Thus the linear functional $\ell_\nu:I^N_{\R}\to\R$ given by $\ell_\nu(c\nu)=c$ satisfies $\ell_\nu(d\beta)=0$ for all $\beta\in I^{N-1}_{\R}$.
\end{lemma}

\begin{proof}
An invariant top-degree form is determined by its value at $o$. Since the top exterior power of $T_o^*X$ is 1-dimensional and $\nu_o\neq0$, we have $I^N_{\R}=\R\nu$.

Let $\beta\in I^{N-1}_{\R}$. There exists a unique
$v\in\mathfrak m$ such that $\beta_o=\iota_v\nu_o$, where $\iota_v$ denotes contraction.
By the $K$-invariance of $\beta_o$ and $\nu_o$, we have
$\iota_{\operatorname{Ad}(k)v}\nu_o=\iota_v\nu_o$
for every $k\in K$. Hence $\operatorname{Ad}(k)v=v$,
and in particular $[v,\mathfrak k]=0$.

Let $\pi:G\to X$ be the natural projection, and let $V$ be the left-invariant vector field on $G$ with $V_e=v$. Then
$$
\pi^*\beta=\iota_V(\pi^*\nu).
$$
Since $\nu$ is a top-degree form on $X$, we have
$d(\pi^*\nu)=0$. Thus Cartan's formula gives
$$
\pi^*(d\beta)
=d\iota_V(\pi^*\nu)
=\mathcal L_V(\pi^*\nu).
$$
Writing $\mathcal L_V$ for the Lie derivative and evaluating at $e$, left invariance gives
$$
\bigl(\mathcal L_V(\pi^*\nu)\bigr)_e(x_1,\ldots,x_N)
=-\sum_{j=1}^N
(\pi^*\nu)_e(x_1,\ldots,[v,x_j],\ldots,x_N),
\qquad x_j\in\mathfrak m.
$$
Since $\pi^*\nu$ is horizontal, the
$\mathfrak k$-components of the brackets do not
contribute. For $B=\operatorname{pr}_{\mathfrak m}\circ\allowbreak\operatorname{ad}_v|_{\mathfrak m}$, the trace identity
$$
\sum_{j=1}^N\nu_o(x_1,\ldots,Bx_j,\ldots,x_N)
=(\tr B)\nu_o(x_1,\ldots,x_N)
$$
therefore yields
\begin{equation}\label{eq:formal-trace}
d\beta=-\tr\bigl(\operatorname{pr}_{\mathfrak m}\circ\operatorname{ad}_v|_{\mathfrak m}\bigr)\nu.
\end{equation}
Since $[v,\mathfrak k]=0$, the trace on the
right-hand side equals $\tr(\operatorname{ad}_v)$,
which vanishes by unimodularity. Therefore
$d\beta=0$, and the assertion about $\ell_\nu$
follows.
\end{proof}

We call the identity in Lemma \ref{lem:formal-stokes}
the \emph{formal Stokes identity}. The functional $\ell_\nu$ extracts a top-degree coefficient rather than an integral over $X$, so it also applies when $X$ is noncompact. For $K=\{e\}$, the identity is $d(\iota_v\nu)=-\tr(\operatorname{ad}_v)\nu=0$.

For the compact-quotient conclusion, we use the following averaging result of Kwong \cite[Proposition A.1]{Kwong}. Its Lie group version can be found in \cite{FG} and \cite[Lemma 3.2]{EFV}.

\begin{lemma}[\cite{Kwong}]\label{lem:average}
Let $G$ be a connected real Lie group, $K\subset G$ compact, and $X=G/K$ a complex $n$-manifold on which $G$ acts holomorphically. Let $\Gamma\subset G$ be discrete and act freely, properly discontinuously and cocompactly on $X$ by left translations, and put $M=X/\Gamma$. Then $G$ is unimodular. For each integer $0\leq q\leq 2n$, there is a real linear map
$$
\mathsf S:\Omega^q_{\R}(M)\longrightarrow I^q_{\R}
$$
which commutes with $d$ and, after complexification, preserves types and commutes with $\partial$ and $\overline\partial$. If $p:X\to M$ is the quotient map and $p^*\rho$ is $G$-invariant, then $\mathsf S\rho=p^*\rho$. The operator preserves positivity of real $(1,1)$-forms. In particular, it takes a Hermitian-symplectic form on $M$ to a $G$-invariant Hermitian-symplectic form on $X$ which descends to $M$.
\end{lemma}

For use with our quotient notation, let $p:X\to M$ be the quotient map, put $\widehat\rho=p^*\rho$, and write $L_a(x)=a\cdot x$. As shown in the proof of \cite[Proposition A.1]{Kwong}, $\Gamma\backslash G$ is compact and has a right-invariant measure $d\mu$ of total mass one. The operator in that proposition is
\begin{equation}\label{eq:average}
(\mathsf S\rho)_x=
\int_{\Gamma\backslash G}(L_a^*\widehat\rho)_x\,d\mu(\Gamma a),
\qquad x\in X.
\end{equation}

We shall use only the properties of $\mathsf S$
stated above. In particular, averaging need not
preserve de Rham cohomology classes. The closed
lift used in the proof below is fixed after
averaging. On a noncompact homogeneous space,
an \emph{invariant Hermitian-symplectic metric} will
mean an invariant metric admitting an invariant
closed real lift.

Let $\mathcal H=I^{1,1}_{\R}=I^{1,1}\cap I^2_{\R}$
be the space of $G$-invariant real $(1,1)$-forms on $X$,
and put $\mathcal P=\{\omega\in\mathcal H:\omega>0\}$.
Following \cite{Kwong}, define the space of
\emph{invariant Aeppli directions} by
\begin{equation}\label{eq:aeppli-directions}
\mathcal A=\{(d\xi)^{1,1}:\xi\in I^1_{\R}\}
=\{\overline\partial\beta+\partial\overline\beta:\beta\in I^{1,0}\}.
\end{equation}
For $\omega_0\in\mathcal P$, its \emph{invariant Aeppli slice} is $\mathcal S(\omega_0)=(\omega_0+\mathcal A)\cap\mathcal P$.

Recall that the Aeppli cohomology in bidegree
$(1,1)$ is defined by
$$
H_A^{1,1}(X)=
\frac{\ker(\partial\overline\partial:\Omega^{1,1}(X)\to\Omega^{2,2}(X))}
{\partial\Omega^{0,1}(X)+\overline\partial\Omega^{1,0}(X)}.
$$
Every element of $\mathcal A$ is Aeppli-exact.
Conversely, on a compact quotient $M=X/\Gamma$,
if a $G$-invariant real $(1,1)$-form descends to
an Aeppli-exact form on $M$, the averaging
operator in Lemma \ref{lem:average} provides
an invariant real 1-form $\xi$ such that $u=(d\xi)^{1,1}$,
where $u$ is the lifted $(1,1)$-form. Without a compact
quotient, our invariant Aeppli slice is defined
using invariant 1-forms, and no identification
with the full space of Aeppli-exact directions
is assumed.

\section{Proof of the main results}\label{sec:proof}

Throughout this section, $G$ is connected and unimodular,
$K$ is compact, and $X=G/K$ has the invariant complex
structure of Theorem \ref{thm:invariant}. Fix a positive
$G$-invariant volume form $\nu$, and let $\ell_\nu$
be the functional defined in Lemma \ref{lem:formal-stokes},
extended complex linearly when necessary. For
$\omega\in\mathcal H$, define
$$
\mathcal V(\omega)=\frac{1}{n!}\ell_\nu(\omega^n).
$$
If $M=X/\Gamma$ is a compact quotient, then
$$
\frac{1}{n!}\int_M\omega^n
=\mathcal V(\omega)\int_M\nu.
$$
Thus $\mathcal V$ differs from the Hermitian volume by a fixed positive factor. The definition of $\mathcal V$
does not require $X$ to be compact.

\begin{lemma}\label{lem:closed-lifts}
Let $\Omega_0$ be an invariant Hermitian-symplectic form, put $\omega_0=\Omega_0^{1,1}$, and set $\mathcal S=\mathcal S(\omega_0)$. For every $\omega\in\mathcal S$ there is an invariant real 1-form $\xi$ for which $\Omega_\xi=\Omega_0+d\xi$ is a closed lift of $\omega$. Every metric in $\mathcal S$ is Hermitian-symplectic and pluriclosed. Moreover, with $C_0=\ell_\nu(\Omega_0^n)/n!$,
\begin{equation}\label{eq:volume-bound}
0<\mathcal V(\omega)\leq C_0
\qquad\text{for every }\omega\in\mathcal S.
\end{equation}
\end{lemma}

\begin{proof}
Let $\omega\in\mathcal S$. By definition, we can write
$\omega=\omega_0+(d\xi)^{1,1}$ for some
$\xi\in I^1_{\R}$. Set $\Omega_\xi=\Omega_0+d\xi$.
Then $\Omega_\xi$ is an invariant closed real lift of $\omega$.
Thus every metric in $\mathcal S$ is Hermitian-symplectic and is
pluriclosed by \eqref{eq:lift-system}.

Since $\Omega_\xi-\Omega_0=d\xi$ and both forms
are closed, we have
\begin{equation}\label{eq:transgression}
\Omega_\xi^n-\Omega_0^n
=d\left(\xi\wedge\sum_{j=0}^{n-1}
\Omega_\xi^{n-1-j}\wedge\Omega_0^j\right).
\end{equation}
The form inside the derivative is invariant of
degree $2n-1$. By Lemma \ref{lem:formal-stokes},
we obtain $\ell_\nu(\Omega_\xi^n)
=\ell_\nu(\Omega_0^n)$. On the other hand, Lemma \ref{lem:volume} gives $\Omega_\xi^n\geq\omega^n>0$. Consequently,
$$
0<\mathcal V(\omega)
\leq\frac{1}{n!}\ell_\nu(\Omega_0^n)
=C_0,
$$
which proves \eqref{eq:volume-bound}.
\end{proof}

Applying this bound to midpoint metrics gives the following estimate.

\begin{lemma}\label{lem:midpoint}
Under the assumptions of Lemma \ref{lem:closed-lifts}, choose an $\omega_0$-unitary frame at $o$, and define $A_\omega$ by $\omega_o=\sqrt{-1}\sum_{i,j}(A_\omega)_{i\bar j}\varphi_i\wedge\overline\varphi_j$ in its dual coframe. Put $V_0=\mathcal V(\omega_0)$. Then
\begin{equation}\label{eq:trace-bound}
\tr A_\omega\leq B_0:=2^n\frac{C_0}{V_0}-1
\qquad(\omega\in\mathcal S).
\end{equation}
Consequently $\mathcal S$ is bounded in $\mathcal H$.
\end{lemma}

\begin{proof}
In the chosen frame, $\omega_0$ has matrix $I_n$, the identity matrix. Since the ratio of invariant top-degree forms is
constant on $X$, we have
\begin{equation}\label{eq:det-volume}
\omega^n=\det(A_\omega)\omega_0^n,
\qquad
\mathcal V(\omega)=\det(A_\omega)V_0.
\end{equation}
For $\omega=\omega_0+(d\xi)^{1,1}\in\mathcal S$,
the midpoint
$\omega_{1/2}=\frac12(\omega_0+\omega)$
lies in $\mathcal S$, with closed lift
$\Omega_0+\frac12d\xi$.
Let $\lambda_1,\ldots,\lambda_n>0$ be the
eigenvalues of $A_\omega$.
By \eqref{eq:volume-bound}, we obtain
$$
\begin{aligned}
\frac{C_0}{V_0}
&\geq\frac{\mathcal V(\omega_{1/2})}{V_0}
 =\det\left(\frac{I_n+A_\omega}{2}\right)=2^{-n}\prod_{j=1}^n(1+\lambda_j)
 \geq2^{-n}\left(1+\sum_{j=1}^n\lambda_j\right).
\end{aligned}
$$
Hence
$$
\tr A_\omega
\leq 2^n\frac{C_0}{V_0}-1,
$$
which proves \eqref{eq:trace-bound}. Since $A_\omega$ is positive definite,
$\|A_\omega\|_{\mathrm{op}}\leq\tr A_\omega$,
where $\|\cdot\|_{\mathrm{op}}$ is the operator norm in the chosen unitary frame.
Evaluation at $o$ embeds $\mathcal H$
into the finite-dimensional space of Hermitian
matrices. Therefore $\mathcal S$ is bounded.
\end{proof}

\begin{remark}\label{rem:semipositive}
Boundedness can also be proved by excluding nonzero semipositive
directions in $\mathcal A$. Suppose that $0\neq u\in\mathcal A$
is semipositive, and let $U$ be its matrix relative to $\omega_0$.
Then $\omega_0+tu\in\mathcal S$ for all $t\geq0$.
Since $U\geq0$ has at least one positive eigenvalue,
$\mathcal V(\omega_0+tu)=\det(I_n+tU)V_0\to\infty$
as $t\to\infty$, contradicting \eqref{eq:volume-bound}.

If $\mathcal S$ were unbounded, we could choose $u_j\in\mathcal A$
with $\omega_0+u_j>0$ and $\|U_j\|_{\mathrm{op}}\to\infty$,
where $U_j$ is the matrix of $u_j$ in the same frame.
After passing to a subsequence,
$U_j/\|U_j\|_{\mathrm{op}}$ converges to a matrix $U$ with
$\|U\|_{\mathrm{op}}=1$. The matrix image of $\mathcal A$
is finite-dimensional and hence closed, so $U$ belongs to this
image. Dividing $I_n+U_j>0$ by $\|U_j\|_{\mathrm{op}}$
and taking the limit gives $U\geq0$, contradicting the
first part. Thus $\mathcal S$ is bounded.

A determinant bound alone does not control the matrix norm,
as $\operatorname{diag}(t,t^{-1})$ shows. The argument above
uses the uniform bound \eqref{eq:volume-bound} on the whole slice.
\end{remark}

The next proposition extends Kwong's volume argument \cite[Proposition 2.1 and Lemma 2.4]{Kwong} with $\ell_\nu$
in place of integration. We give the proof in this form because no compact quotient is assumed here.

\begin{proposition}[{after Kwong,
\cite{Kwong}}]\label{prop:critical}
Let $G$ be a connected unimodular real Lie group,
$K\subset G$ a compact subgroup, and $X=G/K$ a complex
manifold of complex dimension $n\geq2$ with a $G$-invariant
complex structure. Fix a positive invariant volume form $\nu$,
and define $\mathcal H$, $\mathcal P$, $\mathcal A$ and
$\mathcal V$ as above. Let $\omega\in\mathcal P$.
Then $\omega$ is balanced if and only if it is a critical
point of $\mathcal V$ on its invariant Aeppli slice.
If this slice is bounded, it contains a unique critical point,
which is also the unique point at which $\mathcal V$ attains
its maximum.
\end{proposition}

\begin{proof}
For a variation $u\in\mathcal H$, the derivative of $\mathcal V$ at $\omega$ is
$$
(d\mathcal V)_\omega(u)=\frac1{(n-1)!}
             \ell_\nu(u\wedge\omega^{n-1}).
$$
If $u=(d\xi)^{1,1}\in\mathcal A$, only the $(1,1)$-part of $d\xi$ contributes. By Lemma \ref{lem:formal-stokes}, we have
\begin{equation}\label{eq:variation}
(n-1)!(d\mathcal V)_\omega((d\xi)^{1,1})
=\ell_\nu(d\xi\wedge\omega^{n-1})
=\ell_\nu(\xi\wedge d\omega^{n-1}).
\end{equation}
Thus a balanced metric is a critical point of $\mathcal V$ on its invariant Aeppli slice.

Conversely, suppose that $\omega$ is a critical
point, and put $\zeta=d\omega^{n-1}$. Let $g$ be the Hermitian metric associated with $\omega$. The \emph{Hodge star} $*$ of $g$ preserves invariant forms, since the $G$-action is by holomorphic isometries. In particular, $\xi=-*\zeta$ lies in $I^1_{\R}$. With $*$ and $\xi$ fixed at $\omega$, substitute this direction in \eqref{eq:variation} to obtain
$$
0=\ell_\nu((-*\zeta)\wedge\zeta)
  =\ell_\nu(\zeta\wedge*\zeta)
  =|\zeta_o|_g^2\,\ell_\nu(\nu_\omega).
$$
Here $|\cdot|_g$ is the pointwise norm on forms induced by $g$.
Since $\ell_\nu(\nu_\omega)>0$, we have $\zeta_o=0$. By invariance, $\zeta=0$ on $X$, so $\omega$ is balanced.

Now suppose that the slice is bounded. Its closure in $\omega+\mathcal A$ is compact. Every point of its relative boundary has a singular semipositive matrix, and hence zero
volume. Since $\mathcal V(\omega)>0$, the
maximum is attained in the interior.
Such a maximum is necessarily a critical point.

To prove uniqueness, choose a fixed unitary frame at $o$. For every positive form in $\mathcal H$, its top-degree coefficient is a fixed positive multiple of the determinant of its matrix. Thus $\log\mathcal V$ differs from $\log\det$ by an additive constant. Along a line of positive Hermitian matrices $A+tD$, with $D=D^*$ and $D^*={}^t\overline D$,
\begin{equation}\label{eq:hessian}
\frac{d^2}{dt^2}\log\det(A+tD)
=-\tr\bigl((A+tD)^{-1}D(A+tD)^{-1}D\bigr).
\end{equation}
When $D\neq0$, the matrix $B=(A+tD)^{-1/2}D(A+tD)^{-1/2}$ is nonzero and Hermitian, and the trace on the right equals $\tr(B^2)>0$. Thus $\log\mathcal V$ is strictly concave on nonconstant segments. Two distinct critical points would give zero derivatives at both endpoints of such a segment, which is impossible. The critical point, and hence the maximum, is unique.
\end{proof}

\begin{proof}[Proof of Theorem \ref{thm:invariant}]
Let $\Omega_0$ be the given $G$-invariant Hermitian-symplectic
form on $X$, and put $\omega_0=\Omega_0^{1,1}$.
Since $\omega_0>0$, the invariant Aeppli slice
$\mathcal S(\omega_0)$ is nonempty.
By Lemma \ref{lem:closed-lifts}, every metric
in this slice has an invariant closed real lift and is
pluriclosed. The same lemma gives the uniform volume bound
\eqref{eq:volume-bound}, from which Lemma \ref{lem:midpoint}
shows that $\mathcal S(\omega_0)$ is bounded.

By Proposition \ref{prop:critical}, the function $\mathcal V$
attains its maximum at a balanced metric
$\omega_*\in\mathcal S(\omega_0)$.
Since $\omega_*$ is also pluriclosed, we have
$d\omega_*^{n-1}=0$ and
$\partial\overline\partial\omega_*=0$.
Lemma \ref{lem:balanced-skt} therefore gives $d\omega_*=0$.
Thus the positive invariant real $(1,1)$-form $\omega_*$
defines a $G$-invariant K\"ahler metric
$g_*(x,y)=\omega_*(x,Jy)$.

Finally, any balanced metric in $\mathcal S(\omega_0)$
is a critical point of $\mathcal V$ by
Proposition \ref{prop:critical}. The uniqueness of the
critical point implies that this metric must be $\omega_*$.
This proves the uniqueness assertion.
\end{proof}

\begin{proof}[Proof of Corollary \ref{cor:quotient}]
Let $\Omega$ be a Hermitian-symplectic form on $M$, and let
$p:X\to M$ be the quotient map. By Lemma \ref{lem:average},
the group $G$ is unimodular and $\Omega_0=\mathsf S\Omega$
is a $G$-invariant Hermitian-symplectic form on $X$,
where $\mathsf S$ is the averaging operator defined in
that lemma. Theorem \ref{thm:invariant} therefore gives
a $G$-invariant K\"ahler form $\omega_*$ on $X$.

Since $\Gamma\subset G$, the form $\omega_*$ is
$\Gamma$-invariant and hence descends to a real $(1,1)$-form
$\widetilde\omega$ on $M$ satisfying
$p^*\widetilde\omega=\omega_*$. As $p$ is a holomorphic
covering map, $\widetilde\omega$ is positive and closed.
Thus $\widetilde\omega$ defines a K\"ahler metric on $M$
whose lift to $X$ is $G$-invariant.
\end{proof}

\begin{proof}[Proof of Corollary \ref{cor:algebra}]
Let $G$ be the connected simply connected Lie group with
Lie algebra $\mathfrak g$. Extend $J$ and $\Omega$ by left
translation, keeping the same notation. Then $J$ is a
left-invariant integrable complex structure and $\Omega$
is a left-invariant Hermitian-symplectic form on $G$.
Since $\mathfrak g$ is unimodular, so is $G$.
Applying Theorem \ref{thm:invariant} with $K=\{e\}$,
we obtain a left-invariant K\"ahler form $\omega_*$ on $G$.

Let $\widetilde\omega$ be the value of $\omega_*$ at the
identity $e$. It is a positive real $(1,1)$-form on
$(\mathfrak g,J)$, and $d_{\mathfrak g}\widetilde\omega=0$
since $d\omega_*=0$. Thus $\widetilde\omega$ defines the
required K\"ahler inner product on $(\mathfrak g,J)$.
\end{proof}

\begin{remark}\label{rem:nondegenerate}
Let $\omega_j$ be a maximizing sequence for $\mathcal V$
in $\mathcal S(\omega_0)$. Since
$\sup_{\mathcal S(\omega_0)}\mathcal V\geq V_0$, we have
$\mathcal V(\omega_j)\geq V_0/2$ for all sufficiently large $j$.
Thus $\det A_{\omega_j}\geq1/2$.
By \eqref{eq:trace-bound}, every eigenvalue of $A_{\omega_j}$
is at most $B_0$. Their product is at least $1/2$, so the
smallest eigenvalue is at least $(2B_0^{n-1})^{-1}$. Hence
$$
\frac{1}{2B_0^{n-1}}I_n
\leq A_{\omega_j}\leq B_0I_n
$$
for all sufficiently large $j$.

After passing to a subsequence, $A_{\omega_j}$ converges
to a positive definite matrix. Since the invariant forms
belong to a fixed finite-dimensional space, $\omega_j$
converges smoothly on compact subsets of $X$ to a positive
invariant form $\omega_\infty$. Moreover,
$\omega_\infty-\omega_0\in\mathcal A$, since $\mathcal A$
is closed. Thus $\omega_\infty\in\mathcal S(\omega_0)$
and attains the supremum of $\mathcal V$.
When a compact quotient is present, the convergence is
smooth on the whole quotient.

\end{remark}

\begin{remark}\label{rem:classes}
The K\"ahler form $\omega_*$ belongs to
$\mathcal S(\omega_0)$, but need not equal $\omega_0$.
Choose $\xi_*\in I^1_{\R}$ such that
$\omega_*-\omega_0=(d\xi_*)^{1,1}$.
Then $\Omega_*=\Omega_0+d\xi_*$ is an invariant
Hermitian-symplectic form whose $(1,1)$-part is $\omega_*$.
Write $\Omega_*=\omega_*+\alpha_*+\overline{\alpha_*}$,
where $\alpha_*=\Omega_*^{2,0}$.
Since $\Omega_*$ and $\omega_*$ are closed, taking type
components gives $d\alpha_*=0$. Consequently,
$$
[\Omega_0]_{\mathrm{dR}}-[\omega_*]_{\mathrm{dR}}
=
[\alpha_*+\overline{\alpha_*}]_{\mathrm{dR}}.
$$
Thus $\Omega_0$ and $\omega_*$ have the same de Rham class
if and only if $\alpha_*+\overline{\alpha_*}$ is exact.
The theorem does not assert this additional property.

On a compact quotient, the descended average of a
Hermitian-symplectic form need not be cohomologous to
the original form. Compactness is used in
Corollary \ref{cor:quotient} to apply the averaging lemma.
In Theorem \ref{thm:invariant}, the Hermitian-symplectic
form is assumed to be $G$-invariant, and $G$ is assumed
to be unimodular.
\end{remark}

\section{The K\"ahler representative of an invariant slice}\label{sec:representative}

Let $(X=G/K,J)$ be as in Theorem \ref{thm:invariant},
and assume that it admits a $G$-invariant Hermitian-symplectic
form. With $\mathcal H$, $\mathcal P$ and $\mathcal A$
as in Section \ref{sec:invariant}, put
$$
\mathcal L=\{\Omega^{1,1}:\Omega\in I^2_{\R},\ d\Omega=0\},
\qquad
\mathcal Z=\{u\in\mathcal H:du=0\}.
$$
These are real linear subspaces of $\mathcal H$.
The cones $\mathcal L\cap\mathcal P$ and
$\mathcal Z\cap\mathcal P$ consist of the fundamental forms
of invariant Hermitian-symplectic and K\"ahler metrics,
respectively. The first is nonempty by assumption,
and the second by Theorem \ref{thm:invariant}.
We have the following consequence.

\begin{proposition}\label{prop:projection}
Let $G$ be a connected unimodular real Lie group, let $K\subset G$ be a compact subgroup, and let $X=G/K$ have a $G$-invariant integrable complex structure of complex dimension $n\geq2$. Assume that $\mathcal L\cap\mathcal P\neq\varnothing$, with the spaces defined above. Then $\mathcal L=\mathcal Z\oplus\mathcal A$. Let $\mathcal R:\mathcal L\to\mathcal Z$ be the associated linear projection with kernel $\mathcal A$. For every $\omega\in\mathcal L\cap\mathcal P$, the form $\mathcal R(\omega)$ is the unique K\"ahler form in $\mathcal S(\omega)$, and its associated metric is the unique balanced metric in that slice. The projection $\mathcal R$ is independent of the positive invariant volume form used to define $\mathcal V$.
\end{proposition}

\begin{proof}
Choose an invariant K\"ahler form $\kappa$ by Theorem \ref{thm:invariant}. We have $\mathcal Z\subset\mathcal L$ and $\mathcal A\subset\mathcal L$, since $d\xi$ is a closed real lift of $(d\xi)^{1,1}$. If $u\in\mathcal Z\cap\mathcal A$, then $\kappa+tu$ is positive and closed for all sufficiently small real $t$, and belongs to $\mathcal S(\kappa)$. Both $\kappa$ and $\kappa+tu$ are balanced, so uniqueness in Theorem \ref{thm:invariant} gives $tu=0$. Choosing $t\neq0$ yields $u=0$.

For $u\in\mathcal L$, choose an invariant closed real 2-form
$\Psi$ with $\Psi^{1,1}=u$. For sufficiently large $c>0$,
the form $c\kappa+u$ is positive at $o$, and hence on $X$
by invariance. Thus $c\kappa+\Psi$ is an invariant
Hermitian-symplectic form. By Theorem \ref{thm:invariant}, there exist a $G$-invariant
K\"ahler form $\kappa_u$ and $\xi_u\in I^1_{\R}$ such that
$\kappa_u=c\kappa+u+(d\xi_u)^{1,1}$. Therefore
$$
u=(\kappa_u-c\kappa)-(d\xi_u)^{1,1}.
$$
The first term belongs to $\mathcal Z$, since it is an
invariant closed real $(1,1)$-form, and the second belongs
to $\mathcal A$. Hence $\mathcal L\subset\mathcal Z+\mathcal A$.
The reverse inclusion has already been proved, so
$\mathcal L=\mathcal Z+\mathcal A$. Together with
$\mathcal Z\cap\mathcal A=\{0\}$, this proves the
direct sum decomposition.

For $\omega\in\mathcal L\cap\mathcal P$, let $\omega_*$ be its K\"ahler representative given by Theorem \ref{thm:invariant}. Since $\omega_*\in\mathcal Z$ and $\omega-\omega_*\in\mathcal A$, the direct sum decomposition gives $\omega_*=\mathcal R(\omega)$. Uniqueness follows from the same theorem. Finally, $\mathcal R$ is determined by $\mathcal Z$ and $\mathcal A$, neither of which depends on $\nu$.
\end{proof}

The projection gives the following deformation within each invariant Aeppli slice.

\begin{proposition}\label{prop:deformation}
Let $G$ be a connected unimodular real Lie group, let $K\subset G$ be a compact subgroup, and let $X=G/K$ carry a $G$-invariant integrable complex structure of complex dimension $n\geq2$. Assume that $\mathcal L\cap\mathcal P\neq\varnothing$, and let $\mathcal R$ be the projection of Proposition \ref{prop:projection}. Define
$H:[0,1]\times(\mathcal L\cap\mathcal P)\to\mathcal L\cap\mathcal P$ by
$H(t,\omega)=H_t(\omega)=(1-t)\omega+t\mathcal R(\omega)$.
Then $H_t(\omega)\in\mathcal S(\omega)$ for all $0\leq t\leq1$, with $H_0(\omega)=\omega$, $H_1(\omega)=\mathcal R(\omega)$, and $H_t(\kappa)=\kappa$ for every $\kappa\in\mathcal Z\cap\mathcal P$. In particular, $H$ is a strong deformation retraction of the invariant Hermitian-symplectic cone onto the invariant K\"ahler cone. If $\omega$ is not K\"ahler, then $\mathcal V(H_t(\omega))$ is strictly increasing on $[0,1]$; otherwise the path is constant.
\end{proposition}

\begin{proof}
By Proposition \ref{prop:projection}, both $\omega$ and $\mathcal R(\omega)$ are positive, so $H_t(\omega)>0$ for $0\leq t\leq1$. Choose $\xi\in I^1_{\R}$ with $\mathcal R(\omega)-\omega=(d\xi)^{1,1}$. If $\Omega$ is an invariant Hermitian-symplectic form with $\Omega^{1,1}=\omega$, then $\Omega_t=\Omega+t\,d\xi$ is closed and has $(1,1)$-part $H_t(\omega)$. Thus the path remains Hermitian-symplectic and lies in $\mathcal S(\omega)$. Since $\mathcal R$ is linear, $H$ is smooth. The endpoint identities follow from its definition, and $\mathcal R|_{\mathcal Z}$ is the identity, so $H_t$ fixes every K\"ahler form in $\mathcal Z\cap\mathcal P$.

If $\omega$ is not K\"ahler, then
$\mathcal R(\omega)\neq\omega$, so the segment is nonconstant.
Put $f(t)=\log\mathcal V(H_t(\omega))$.
By \eqref{eq:hessian}, we have $f''(t)<0$.
The K\"ahler form $\mathcal R(\omega)$ is a critical point
of $\mathcal V$ on $\mathcal S(\omega)$ by
Proposition \ref{prop:critical}. Since
$\mathcal R(\omega)-\omega\in\mathcal A$, it follows that
$f'(1)=0$. Therefore $f'(t)>0$ for $0\leq t<1$,
and $\mathcal V(H_t(\omega))$ is strictly increasing
on $[0,1]$.

If $\omega$ is K\"ahler, then $\mathcal R(\omega)=\omega$,
so the path is constant.
\end{proof}

\section{Examples}\label{sec:examples}

The first example describes a nontrivial invariant slice; the second shows that a pluriclosed slice need not be bounded.

\begin{example}
On $\R_s\times\R_t\times\C_w$, consider the Lie group $L$ with multiplication
\begin{equation}\label{eq:group-example}
(s,t,w)(s',t',w')=(s+s',t+t',w+e^{2\sqrt{-1}s}w').
\end{equation}
The usual complex structure in $(s+\sqrt{-1}t,w)$ is left invariant, since left translations are complex affine. The Lie algebra has a real basis $S,T,U,V$ with nonzero brackets $[S,U]=2V$ and $[S,V]=-2U$. In particular, $L$ is nonabelian and unimodular. Its left-invariant coframe
$$
\varphi_1=ds+\sqrt{-1}\,dt,
\qquad
\varphi_2=e^{-2\sqrt{-1}s}dw
$$
satisfies
\begin{equation}\label{eq:model-structure}
d\varphi_1=0,
\qquad
d\varphi_2=-\sqrt{-1}(\varphi_1+\overline\varphi_1)\wedge\varphi_2.
\end{equation}
The discrete subgroup $\Gamma_L=(\pi\Z)\times\Z\times(\Z+\sqrt{-1}\Z)$ acts by ordinary translations, since $e^{2\sqrt{-1}\pi k}=1$. In the complex coordinates $(s+\sqrt{-1}t,w)$, this action is translation by the lattice $(\pi\Z+\sqrt{-1}\Z)\times(\Z+\sqrt{-1}\Z)$. Thus $L/\Gamma_L$ is a compact complex two-torus.

Let $G_0=\operatorname{PSL}(2,\R)$ and let $K_0$ be the image of $\operatorname{SO}(2)$ in $G_0$, so that $G_0/K_0=\mathbb D$ is the unit disk. Choose a torsion-free cocompact lattice $\Lambda\subset G_0$, and put $\Sigma=\mathbb D/\Lambda$. This is a compact Riemann surface of genus at least two. With $G=G_0\times L$, $K=K_0\times\{e\}$ and $\Gamma=\Lambda\times\Gamma_L$, we have
$$
X=\mathbb D\times L,
\qquad
M=X/\Gamma=\Sigma\times(L/\Gamma_L).
$$
The group $G$ is noncompact and unimodular, and $K$ is compact of dimension one. A stabilizer for the $\Lambda$-action is a discrete subgroup of a compact group and hence finite. It is trivial by torsion-freeness, so the product action is free. The $G$-action on $G/K$ is proper because $K$ is compact, and its restriction to the discrete subgroup $\Gamma$ is properly discontinuous. The two factor quotients are compact, so the product quotient is compact as well.

For the invariant hyperbolic K\"ahler form $\omega_{\mathbb D}$, set
$$
\kappa=\omega_{\mathbb D}
 +\sqrt{-1}\varphi_1\wedge\overline\varphi_1
 +\sqrt{-1}\varphi_2\wedge\overline\varphi_2.
$$
Forms on the factors are identified with their pullbacks to $X$. By \eqref{eq:model-structure}, $d(\varphi_2\wedge\overline\varphi_2)=0$, so $\kappa$ is K\"ahler. For $z\in\C$, put $\xi_z=z\varphi_2+\overline z\,\overline\varphi_2$, $\Omega_z=\kappa+d\xi_z$ and $\omega_z=\Omega_z^{1,1}$. We compute
\begin{equation}\label{eq:model-components}
\begin{aligned}
(d\xi_z)^{1,1}
 &=\sqrt{-1}z\varphi_2\wedge\overline\varphi_1
   +\sqrt{-1}\overline z\,\varphi_1\wedge\overline\varphi_2,\\
\alpha_z:=(d\xi_z)^{2,0}
 &=-\sqrt{-1}z\varphi_1\wedge\varphi_2.
\end{aligned}
\end{equation}
The matrix of $\omega_z$ on $L$ is $\left(\begin{smallmatrix}1&\overline z\\z&1\end{smallmatrix}\right)$, so $\omega_z>0$ exactly when $|z|<1$. Although $\Omega_z$ is closed for all $z$,
$$
d\omega_z
=z\varphi_1\wedge\varphi_2\wedge\overline\varphi_1
 +\overline z\,\varphi_1\wedge\overline\varphi_1\wedge\overline\varphi_2
$$
is nonzero for $z\neq0$.

Use a subscript $L$ for restriction to the $L$ factor, so $\kappa_L=\sqrt{-1}\varphi_1\wedge\overline\varphi_1+\sqrt{-1}\varphi_2\wedge\overline\varphi_2$. Then $\omega_{z,L}^2=(1-|z|^2)\kappa_L^2$ and $2\alpha_z\wedge\overline\alpha_z=|z|^2\kappa_L^2$, giving $\Omega_{z,L}^2=\kappa_L^2$. Since $\dim_{\C}\mathbb D=1$,
\begin{equation}\label{eq:model-volumes}
\Omega_z^3=\kappa^3,
\qquad
\omega_z^3=(1-|z|^2)\kappa^3.
\end{equation}
There are no nonzero $G_0$-invariant real 1-forms on $\mathbb D$, since the isotropy acts by rotations. Hence every $G$-invariant real 1-form on $X$ is pulled back from a left-invariant real 1-form on $L$. The latter are spanned by the real and imaginary parts of $\varphi_1,\varphi_2$. Since $d\varphi_1=0$, \eqref{eq:model-components} parametrizes all directions in $\mathcal A$. Therefore
$$
\mathcal S(\kappa)=\{\omega_z:|z|<1\}.
$$
The unique volume maximum is attained at $z=0$, with $\mathcal R(\omega_z)=\kappa$. The deformation in Proposition \ref{prop:deformation} is $z\mapsto(1-t)z$.

The compact quotient $M=\Sigma\times(L/\Gamma_L)$ admits no transitive holomorphic action of a connected Lie group. Indeed, the component of a holomorphic vector field along $\Sigma$, restricted to any $\Sigma\times\{p\}$, is a section of $T\Sigma$. Since $\deg T\Sigma<0$, this component vanishes. No connected group of holomorphic automorphisms can therefore act transitively. This K\"ahler manifold carries the non-K\"ahler Hermitian-symplectic metrics $\omega_z$ for $0<|z|<1$.
\end{example}

\begin{example}
Choose a real left-invariant coframe $a_0,a_1,a_2,a_3$ on $\operatorname{SU}(2)\times\R$ with
$$
da_0=-2a_1\wedge a_2,\quad
da_1=-2a_2\wedge a_0,\quad
da_2=-2a_0\wedge a_1,\quad
da_3=0.
$$
For $\varphi_1=a_0+\sqrt{-1}a_3$ and $\varphi_2=a_1+\sqrt{-1}a_2$, we have
$$
d\varphi_1=-\sqrt{-1}\varphi_2\wedge\overline\varphi_2,
\qquad
d\varphi_2=-\sqrt{-1}(\varphi_1+\overline\varphi_1)\wedge\varphi_2.
$$
Declaring these forms to have type $(1,0)$ defines an invariant integrable complex structure. Let $\Gamma=\{e\}\times\Z\subset\operatorname{SU}(2)\times\R$, and put $M=(\operatorname{SU}(2)\times\R)/\Gamma$. The complex structure and the left-invariant forms descend to this compact quotient. We use the same notation for the descended forms.

Put $\omega_0=\sqrt{-1}\sum_{j=1}^2\varphi_j\wedge\overline\varphi_j$. By unimodularity and Lemma \ref{lem:formal-stokes}, extended complex linearly, every invariant three-form is closed. Applied to $\partial\omega_0$, this gives $\partial\overline\partial\omega_0=0$.

The form $u=-da_0=\sqrt{-1}\varphi_2\wedge\overline\varphi_2$ is exact and semipositive of rank one at every point. Hence $\omega_t=\omega_0+tu$ belongs to $\mathcal S(\omega_0)$ for all $t\geq0$, and $\omega_t^2=(1+t)\omega_0^2$. Thus this pluriclosed slice is unbounded.

Suppose that $M$ admits a Hermitian-symplectic form $\Omega$. Since $a_0$ is globally defined on $M$,
$$
u\wedge\Omega=d(-a_0\wedge\Omega),
$$
and hence $\int_Mu\wedge\Omega=0$ by Stokes' theorem. On the other hand, since $\dim_{\C}M=2$,
$$
u\wedge\Omega=u\wedge\Omega^{1,1}>0,
$$
because $u$ has rank one everywhere and $\Omega^{1,1}$ is positive definite. This contradiction excludes any Hermitian-symplectic structure on $M$. It is a direct application of \cite[Lemma 2]{CaZ} and does not use Corollary \ref{cor:quotient}.
\end{example}

\begin{remark}
Without unimodularity the formal Stokes identity can fail. On $\mathfrak{aff}(\R)\oplus\R^2$, where $\mathfrak{aff}(\R)$ is the Lie algebra of the affine group of the real line, choose a real coframe $b_1,b_2,b_3,b_4$ with $db_2=-b_1\wedge b_2$ and all other differentials zero. Then $d(b_2\wedge b_3\wedge b_4)=-b_1\wedge b_2\wedge b_3\wedge b_4$. Thus the formal Stokes identity fails in this example. This does not show that a nonunimodular analogue of Theorem \ref{thm:invariant} is false. Integrability is also essential to the argument through \eqref{eq:lift-system} and the pluriclosed condition.
\end{remark}

\noindent\textbf{Acknowledgments.}
We would like to express our sincere gratitude to Professor Fangyang
Zheng for his insightful guidance, valuable discussions, and continued support throughout the course of this research.

\vspace{0.3cm}

\noindent\textbf{Generative AI disclosure.}
The authors used OpenAI ChatGPT as an assisting tool in the
development of the paper. The authors take full responsibility for the contents of the paper.

\vspace{0.3cm}

\noindent\textbf{Declaration on competing interests.}
All authors declare that there are no competing interests for this article.

\end{document}